\documentclass[reqno,11pt,twoside]{article}
\usepackage[hmargin=1.25in,vmargin=1.25in, a4paper, centering]{geometry}
\usepackage{fourier}
\usepackage{ebgaramond}
\usepackage{amssymb, amsmath, amsthm}
\usepackage{graphicx}
\usepackage{xcolor}
\theoremstyle{definition}
\newtheorem{defi}{Definition}[section]

\theoremstyle{remark}
\newtheorem{remark}[defi]{Remark}
\theoremstyle{plain}
\newtheorem{theorem}[defi]{Theorem}
 \newtheorem{prop}[defi]{Proposition}

\newtheorem{cor}[defi]{Corollary}

\newtheorem*{openproblem}{Open Problem}
\newcommand{\R}{\mathbb{R}}
\newcommand{\di}{\operatorname{div}}
\newcommand{\grad}{\nabla}

\newcommand{\vol}{\operatorname{vol}}

\newcommand{\Dir}{{\mathrm{D}}}
\newcommand{\Neu}{{\mathrm{N}}}
\newcommand{\Rob}{{\mathrm{R}}}

\newcommand{\de}{\, \mathrm{d}}
\newcommand{\pardiff}[2]{\frac{\partial #1}{\partial #2}}
\newcommand{\abs}[1]{\left| #1 \right|}

\renewcommand{\epsilon}{\varepsilon}
\newcommand{\Spec}{\operatorname{Spec}}
\newcommand{\DtN}{\mathcal{D}}
\newcommand{\CA}{\mathcal{A}}
\newcommand{\CB}{\mathcal{B}}
\newcommand{\CC}{\mathcal{C}}
\newcommand{\DN}{\mathcal{D}}
\newcommand{\CE}{\mathcal{E}}
\newcommand{\CH}{\mathcal{H}}
\newcommand{\CL}{\mathcal{L}}
\DeclareMathOperator{\tr}{tr}
\DeclareMathOperator{\Dom}{Dom}
\DeclareMathOperator{\Span}{Span}

\newcommand{\myscal}[1]{\left(#1\right)}
\newcommand{\myinn}[1]{\left\langle#1\right\rangle}
\newcommand{\II}{\mathcal{Q}}

\usepackage{fancyhdr}
\renewcommand\footnotemark{}
\usepackage[sf,sl,outermarks]{titlesec}
\titleformat{\section}
{\normalfont\large\bfseries}
{\filcenter\S\thesection.}{1ex}{\filcenter}
\titleformat{\subsection}[runin]
{\normalfont\normalsize\bfseries}{\S\thesubsection.}{1ex}{}[.]
\usepackage[colorlinks,allcolors=blue]{hyperref}

\begin{document}
\title{%
The Dirichlet-to-Neumann map, the  boundary Laplacian, and H\"ormander's rediscovered manuscript%
}
\author{%
Alexandre Girouard\hspace{-3ex}
\thanks{%
\textbf{A. G.: }D\'epartement de math\'ematiques et de statistique, Pavillon Alexandre-Vachon, Universit\'e Laval, 
Qu\'ebec, QC, G1V 0A6, Canada; 
\href{mailto:alexandre.girouard@mat.ulaval.ca}{\nolinkurl{alexandre.girouard@mat.ulaval.ca}}; \url{http://archimede.mat.ulaval.ca/agirouard/}%
}
\and
Mikhail Karpukhin\hspace{-3ex}
\thanks{%
\textbf{M. K.: }Mathematics 253-37, California Institute of Technology, 
Pasadena, CA 91125, USA; 
\href{mailto:mikhailk@caltech.edu}{\nolinkurl{mikhailk@caltech.edu}}; \url{http://sites.google.com/view/mkarpukh/home}%
}
\and 
Michael Levitin\hspace{-3ex}
\thanks{%
\textbf{M. L.: }Department of Mathematics and Statistics, University of Reading, 
Pepper Lane, Whiteknights, Reading RG6 6AX, UK;
\href{mailto:M.Levitin@reading.ac.uk}{\nolinkurl{M.Levitin@reading.ac.uk}}; \url{http://www.michaellevitin.net}%
}
\and 
Iosif Polterovich
\thanks{%
\textbf{I. P.: }D\'e\-par\-te\-ment de math\'ematiques et de statistique, Univer\-sit\'e de Mont\-r\'eal 
CP 6128 succ Centre-Ville, Mont\-r\'eal QC  H3C 3J7, Canada;
\href{mailto:iossif@dms.umontreal.ca}{\nolinkurl{iossif@dms.umontreal.ca}}; \url{http://www.dms.umontreal.ca/\~iossif}%
}
}
\date{
\emph{Dedicated to the memory of Misha Shubin}} 
\maketitle
\begin{abstract}  How close is  the  Dirichlet-to-Neumann (DtN) map to the square root of the 
corresponding boundary Laplacian?  This  question has been actively investigated in recent years. Somewhat surprisingly,  a lot of techniques involved can be traced back to a newly rediscovered  manuscript of H\"ormander from the 1950s. We  present  H\"{o}rmander's approach and   its  applications, with an emphasis on eigenvalue estimates and spectral asymptotics.
In particular,  we  obtain results for  the DtN maps on non-smooth boundaries in the Riemannian setting, the  DtN operators for the Helmholtz equation and the DtN operators on differential forms. 
\end{abstract}
\tableofcontents
\section{Introduction and main results}\label{sec:intro}
\subsection{The Steklov spectrum and the Dirichlet-to-Neumann map}
Let $\Omega$ be a bounded domain in a complete smooth Riemannian manifold $X$   of dimension $d \geq 2$ and let $\Delta$ be the (positive) Laplacian on $\Omega$. 
Assume that the  boundary $\partial \Omega = M$ is Lipschitz. 
Consider the {\it Steklov eigenvalue problem} on $\Omega$:
\begin{equation}
\label{eq:Steklovproblem}
\begin{cases}
\Delta U = 0 & \text{in } \Omega; \\
\pardiff{U}{n} = \sigma U & \text{on } M.
\end{cases}
\end{equation}
We refer to  \cite{KK+14} for a historical discussion and to \cite{GiPo17} for a survey on this eigenvalue problem and related questions in spectral geometry.

The Steklov spectrum is discrete, and the eigenvalues form a sequence
$ 0=\sigma_1<\sigma_2\le\sigma_3\le\dots \nearrow +\infty$.
Alternatively, the Steklov eigenvalues can be viewed as the eigenvalues of the \emph{Dirichlet-to-Neumann map},
\[
\mathcal{D}:
\begin{aligned}H^{1/2}(M)&\to H^{-1/2}(M),\\
u &\mapsto  \partial_n U,
\end{aligned}
\]
where $\partial_n U:=\pardiff{U}{n}=\myinn{(\nabla U)|_M,n}$, $n$ is the unit outward normal vector field along $M$,   and the solution  $U:=U_u$  of $\Delta U=0$ in $\Omega$, $U|_M=u$,  is the  unique \emph{harmonic extension} of the function  $u$  from the boundary into $\Omega$ (we refer to \cite[Section 5]{MiTa99} for  uniqueness and existence results for the solutions of the Dirichlet problem on Lipschitz domains in Riemannian manifolds).
The eigenfunctions of $\DN$ are the restrictions of the Steklov eigenfunctions to the boundary, and form
an orthogonal basis   of $L^2(M)$.

\subsection{The Dirichlet-to-Neumann map and the boundary Laplacian}\label{subsec:dtnlap}
The goal of this paper is to 
explore the links between the Dirichlet-to-Neumann map $\DtN$ and the boundary Laplacian $\Delta_M$. 

If  $\Omega$  has a smooth  boundary $M$, $\mathcal{D}$ is a self-adjoint elliptic pseudodifferential operator of order one on $M$ with the same principal symbol 
as   $\sqrt{\Delta_{M}}$,  i.e.\ the square root of the (positive)  boundary Laplacian. In this case, the following sharp Weyl's law holds for the Steklov eigenvalues (see, for instance, \cite{GPPS14}):
\begin{equation}
\label{eq:weylawcount}
N(\sigma)=\#(\sigma_k<\sigma) = \frac{\vol(\mathbb{B}^{d-1}) \vol(M)}{(2\pi)^{d-1}} \sigma^{d-1} + O\left(\sigma^{d-2}\right),
\end{equation}
or, equivalently,
\begin{equation}
\label{eq:weylaw}
\sigma_k=2\pi \left(\frac{k}{\vol(\mathbb{B}^{d-1}) \vol(M)}\right)^\frac{1}{d-1} + O(1).
\end{equation}
Let us denote by  $ 0=\lambda_1\le\lambda_2\le\lambda_3\le\dots \nearrow +\infty$ the eigenvalues of  the boundary Laplacian $\Delta_{M}$.
Comparing \eqref{eq:weylaw} with the sharp Weyl's law for  the boundary Laplacian (see \cite[Chapter III]{Shubin}), we obtain
\begin{equation}\label{eq:sigmalambda}
\left|\sigma_k - \sqrt{\lambda_k}\right| < C, \quad k\ge 1,
\end{equation}
for some constant $C>0$ depending on $\Omega$.
\begin{remark}
\label{rem:clar}
As we show later, see Remarks \ref{rem:distance} and \ref{rem:neigh}, the constant $C$ in \eqref{eq:sigmalambda} depends only on the geometry of 
$\Omega$ in an arbitrary small neighbourhood of $M$. 
Moreover, our approach yields  \eqref{eq:sigmalambda}  under significantly weaker regularity assumptions on $M$, see Theorem \ref{thm:Hormbound}.%
\end{remark}

If $d=2$, the asymptotic results above can be made much more precise since the Dirichlet-to-Neumann map acts, in this case, on a one-dimensional boundary. 
\begin{theorem}[\cite{GPPS14}]\label{GPPS}
For any smooth surface $\Omega$ with   $m$ boundary components of lengths $l_1, \dots l_m$,  
\begin{equation}\label{eq:gpps}
\sigma_k - \sqrt{\lambda_k} =  o(k^{-\infty}), \quad k\to+\infty,
\end{equation}
where $\lambda_k$ is the $k$-th eigenvalue of the Laplacian on the disjoint  union of circles  of  lengths  $l_1, \dots l_m$.
\end{theorem}
For simply connected planar domains, this result was proved by Rozenblyum \cite{Ros79} and rediscovered by Guillemin--Melrose, see \cite{Ed93}. 

Recall that the Steklov eigenvalues of the unit disk are $0,1,1,2,2,\dots, k,k, \dots$, and hence $\sigma_k=\sqrt{\lambda_k}$ for all $k\ge1$.
Moreover, the boundary Laplacian $\Delta_{\mathbb{S}^1}$ coincides with the square of the Dirichlet-to-Neumann map on the disk.
Our first result shows that such an equality of operators occurs if and only if the Euclidean domain is a disk.
\begin{theorem}\label{thm:rigid}
Let $\Omega \subset \mathbb{R}^d$ be a smooth bounded Euclidean domain, $d \ge 2$.  If \, $\mathcal{D}=\sqrt{\Delta_M}$  then $d=2$ and $\Omega$ is a disk. Moreover, if  $M$ is connected and $d\ge 3$, $\mathcal{D}$ commutes with $\Delta_M$ if and only if  $\Omega$ is a ball.
\end{theorem}
Note that  the ``if'' part of the second statement follows from the well-known fact that the eigenfunctions of the Dirichlet-to-Neumann map on the sphere are precisely the spherical harmonics.
Theorem \ref{thm:rigid} is proved in \S\ref{sec:rigid} using a combination of symbolic calculus and some simple arguments from differential geometry.
\subsection{H\"ormander's identity and its applications} \label{subsec:horm}
The inequalities between the Steklov and Laplace eigenvalues discussed in the previous subsection were  obtained using pseudodifferential techniques for domains with  smooth boundaries.  It is more efficient to use other  techniques  in order to extend these results to non-smooth domains, as well as to characterise the difference between the Dirichlet-to-Neumann operator and the boundary Laplacian in geometric terms. These questions have been addressed in a series of recent papers starting with the work of Provenzano--Stubbe \cite{PrSt18, HaSi19, Xi18, CGH18}.  
\begin{remark}\label{rem:Pohid}
The approach used in these papers is based on the so-called Pohozhaev's identity \cite{Poh65}, which in turn is an application of the method of multipliers  going back to Rellich (see \cite[p. 205]{CWGLS12} for a discussion). One of the objectives of the present paper is to show  that, surprisingly enough,  these results go back to an old unpublished work L. H\"ormander \cite{Hor18} that was originally written in 1950s (see also \cite{Hor54} where an identity similar to Pohozhaev's has been obtained). 
\end{remark}
In what follows, we assume that $\Omega$ has a $C^{1,1}$ boundary, i.e. at each point of the boundary there exists a smooth coordinate chart on the ambient manifold $X$  in which the image of  $\partial \Omega$ coincides with a graph  of a $C^{1,1}$ function. Note that in this case the outward  unit normal vector field  on $\partial \Omega$ is Lipschitz continuous and the induced Riemannian metric on the boundary $M$ is Lipschitz. We let
\[
\myscal{\ ,\ }_\Omega\qquad\myscal{\ ,\ }_M,\qquad\|\ \|_\Omega,\qquad\|\ \|_M
\]
denote the inner products and norms in $L^2(\Omega)$ and $L^2(M)$, respectively.  
\begin{theorem}\label{thm:Hormid}
Let  $X$ be a complete smooth Riemannian manifold, and let $\Omega \subset X$ be a  bounded domain with a $C^{1,1}$ boundary.
Let  $u\in H^{1}(M)$, let $U$ be the harmonic extension of $u$,  and let $F$  be a Lipschitz vector field on $\overline{\Omega}$, such that a restriction of $F$ to $M$ coincides with the outward  unit normal $n$.  Then
\begin{equation}\label{eq:Hormid}
\myscal{\mathcal{D} u,\mathcal{D} u}_M-\myscal{\nabla_{M} u,\nabla_M u}_M=\int_{\Omega}\left(2\myinn{\nabla_{\grad U}F, \grad U}-|\grad U|^2 \di F\right)\, \de v_\Omega,
\end{equation}
where $\nabla_{\nabla U} F$ denotes the covariant derivative  of $F$ in the direction of $\nabla U$ and $dv_\Omega$ is the Riemannian measure  on $\Omega$.
\end{theorem}
\begin{remark}\label{rem:Lipcontvf}
Here and further on, we understand the covariant derivatives of a Lipschitz vector field  as elements of $L^\infty(\Omega)$.
\end{remark}
\begin{remark} The integrand in the right-hand side of \eqref{eq:Hormid} can be expressed in terms of   {\it H\"ormander's energy tensor} defined in \cite{Hor54, Hor18}.
If $\Omega$ is a Euclidean domain then $\myinn{\nabla_{\grad U}F,\grad U} = DF[\nabla U, \nabla U]$, where $DF$ is the Jacobian of $F$.
\end{remark}

The integrand in the right-hand side of \eqref{eq:Hormid} is a quadratic form in $\nabla U$ with bounded coefficients, since $F$ is Lipschitz. Hence, for some $C>0$  depending only on  $F$, we obtain (see Corollary 
\ref{cor:hormcor1})
\begin{equation}\label{eq:Horcor}
\left|\myscal{\mathcal{D} u,\mathcal{D} u}_M - \myscal{\nabla_{M} u,\nabla_M u}_M\right| \le C \int_\Omega |\nabla U|^2\, \de v_\Omega.
\end{equation}

\begin{remark}\label{rem:distance}
Let $d_M(\cdot)$ be the signed distance function to the boundary $M$ defined on $X$,  positive inside $\Omega$. Consider the vector field $F:=\nabla\left(d_M \chi\right)$, where $\chi(\cdot)$ is a smooth cut-off function equal to one near $M$ and zero outside a small neighbourhood of $M$. Then $F$ satisfies the assumptions of Theorem \ref{thm:Hormid}, see \cite[Section 3]{DZ98} and 
\cite[Subsection 5.3]{PrSt18}. Note that  while \cite[Theorem 3.1]{DZ98} is presented in the Euclidean setting, the required statement can be  adapted to the Riemannian case in a straightforward way, cf. \cite[Section 2]{CGH18}. 
In particular, 
this implies that the constant $C$ in \eqref{eq:Horcor} depends only  on the geometry of $\Omega$ in an arbitrary small neighbourhood of $M$, cf. Remark \ref{rem:clar}.  Moreover,  if  $\Omega$ is a smooth Riemannian manifold with boundary, it  was shown in \cite{CGH18} that
the  constant  $C$ can be estimated in terms of the rolling radius  of $\Omega$,  bounds on the sectional curvatures in a 
tubular neighbourhood of $M$, and the principal curvatures of $M$. In fact, obtaining an explicit control on $C$ in terms of  these geometric quantities was one of the main  results of \cite{PrSt18, Xi18, CGH18}. 
\end{remark}

Using the variational principle, one can deduce from \eqref{eq:Horcor} the following statement. Under $C^2$ regularity assumptions it  was first proved in the Euclidean setting in \cite[Theorem 1.7]{PrSt18}, and then in the smooth Riemannian setting in \cite[Theorem 3]{CGH18}.
\begin{theorem}\label{thm:Hormbound}
Let $X$ be a complete smooth Riemannian manifold, and let $\Omega \subset X$ be a  bounded domain with a $C^{1,1}$ boundary. Then 
\begin{equation}\label{eq:steklap}
\left|\sigma_k - \sqrt{\lambda_k}\right| \le C
\end{equation}
holds  for all $k\in\mathbb{N}$ with the same constant $C$ as in \eqref{eq:Horcor}.
\end{theorem}
The proofs of Theorems \ref{thm:Hormid} and \ref{thm:Hormbound} are presented in  \S\ref{sec:identities}.

\begin{remark}\label{rem:neigh} As shown in Remark \ref{rem:distance}, we can choose the constant $C$ in the right-hand side of \eqref{eq:steklap} to depend only on an arbitrarily small neighbourhood of the boundary. 
\end{remark}

Theorem \ref{thm:Hormbound} together with the results of  \cite{Ziel99,Iv00} can be applied to extend Weyl's asymptotics \eqref{eq:weylaw} to domains with non-smooth boundaries.
\begin{theorem}\label{thm:rough1}
Let $X$ be a complete smooth Riemannian manifold of dimension $d$,  and let $\Omega \subset X$ be a bounded domain with $C^{2,\alpha}$ boundary for some $\alpha>0$.  Then the sharp 
Weyl asymptotic formula \eqref{eq:weylawcount} holds for the Steklov eigenvalues on $\Omega$. 
\end{theorem}
Moreover, if $d=2$, the regularity assumption in Theorem \ref{thm:rough1} can be improved even further.
\begin{theorem} \label{thm:rough2}
Let $X$ be a complete smooth Riemannian surface and let $\Omega \subset X$ be a bounded domain with a $C^{1,1}$ boundary. Then the Steklov eigenvalues $\sigma_k$ of $\Omega$ satisfy
Weyl's asymptotics 
\begin{equation}
\label{eq:asymptwo}
\sigma_k =\frac{\pi k}{\abs{\partial \Omega}} + O(1)
\end{equation}
 as  
$k \to \infty$.
\end{theorem}
Theorems \ref{thm:rough1} and \ref{thm:rough2} are proved in \S\ref{subsec:asymp}.

\subsection{Plan of the paper} The paper is organised as follows. In \S\ref{sec:rigid} we prove Theorem \ref{thm:rigid} and discuss some related open questions. 
In \S\ref{sec:identities} we obtain Proposition \ref{lem:abstract} which is an abstract form of Theorem \ref{thm:Hormid}, and then prove the main results of the present paper stated in \S\ref{subsec:horm}.  The rest of the paper is concerned with extensions and applications of Theorem \ref{thm:Hormid} to other settings. In \S\ref{sec:lambda} we consider the Dirichlet-to-Neumann operators associated with the Helmholtz equation $\Delta U = \mu U$ in $\Omega$.  In particular, we get a generalisation  of the bound \eqref{eq:steklap} which is uniform for 
all $\mu \le 0$, see Theorem~\ref{thm:Hormboundgen}.  Interestingly enough, such a uniform estimate does not hold for domains with corners, see Proposition \ref{prop:polyg}. In a way, this observation shows that the $C^{1,1}$ regularity assumption which is needed for the proof of Theorem \ref{thm:Hormboundgen}  can not be relaxed by too much. Some difficulties 
arising in the case $\mu>0$ are also discussed. In \S\ref{sec:forms} we extend the results of \S\ref{subsec:horm} to the Dirichlet-to-Neumann operator on differential forms as defined in \cite{Kar19}. In particular, we prove an analogue of the estimate \eqref{eq:steklap} comparing the eigenvalues of the DtN operator on co-closed forms with the corresponding eigenvalues of  the Hodge Laplacian. As a consequence, we obtain Weyl's law for the DtN operator on forms which has not been known previously.
\subsection*{Acknowledgements}\addcontentsline{toc}{subsection}{Acknowledgements}  
The authors are grateful to Graham Cox, Asma Hassannezhad, Konstantin Pankrashkin, David Sher and  Alexander Strohmaier for helpful discussions. A.G. and I.P. would also like to thank Yakar Kannai for providing them with a copy of the original H\"ormander's manuscript  before it was  published as \cite{Hor18}. The research of A.G. and I.P. is partially supported by NSERC, as well as by FRQNT team grant \#283055. M.K. is partially supported by NSF grant DMS-1363432.
\section{Commutators and rigidity}\label{sec:rigid}
\subsection{Proof of Theorem \ref{thm:rigid}}
We start by proving the second part of the theorem. Since  the boundary of $\Omega$ is smooth, the Dirichlet-to-Neumann operator is an elliptic pseudodifferential operator of order one, and it is related to the boundary Laplacian by
\[
\mathcal{D}=\sqrt{\Delta_M}+B,
\]
where 
$B$ is a $0$-th order pseudodifferential operator on $M$ with  the principal symbol 
\begin{equation*}
\beta(x,\xi)=\frac{1}{2}\left(\frac{\II(\xi,\xi)}{|\xi|^2}-H\right).
\end{equation*}
Here, $\II(\xi,\xi)$ is the second fundamental form of $M$ in $\Omega$ and $H$ is the mean curvature of $M$, that is the trace of the second fundamental form.
We refer to \cite[Chapter 12, Proposition C1]{Taylor} and \cite[formula (4.1.2)]{PoSh15} for the derivation of this formula.
Note that since the subprincipal symbol of $\sqrt{\Delta_M}$ vanishes (see, for instance, \cite{Gu78}), $\beta(x,\xi)$  is  the subprincipal symbol of $\DtN$.

Consider the commutator $T=[\Delta_M,\mathcal{D}]$. Since the principal  symbol $|\xi|^2$ of the Laplacian commutes with the principal  symbol $|\xi|$ of $\DtN$, the order two part of the symbol of the commutator vanishes, and  the operator $T$ is of order one. Up to a  
constant multiple, its  principal symbol  is given by the Poisson bracket $\{|\xi|^2, \beta(x,\xi)\}$ (see \cite[Appendix]{Gu78}).
To compute this expression,  we use the boundary normal coordinates at a given point $p\in M$, so that the Riemannian metric  satisfies $g_{ij}(p)=\delta_{ij}$,
and the first order derivatives of the metric tensor vanish at $p$. It follows that the Poisson bracket evaluated at $x=p$ is given by
\[
  \{|\xi|^2,\beta(x,\xi)\}
  =\sum_k\frac{\partial}{\partial\xi_k}g_{ij}\xi^i\xi^j\frac{\partial}{\partial x_k}\beta(x,\xi)
  -
  \sum_k\frac{\partial}{\partial\xi_k}\beta(x,\xi)\underbrace{\frac{\partial}{\partial x_k}g_{ij}\xi^i\xi^j}_0.
\]
The hypothesis that $\Delta_M$ and $\mathcal{D}$ commute therefore imply the following identity at $x=p\in M$:
\begin{equation}\label{eq:symbolmanip1}
  0=2\sum_k\xi^k\frac{\partial}{\partial x_k}\beta(x,\xi)\left|_{x=p}\right.=\sum_k\xi^k\left(\frac{\partial\II_{ij}}{\partial x_k}\frac{\xi^i\xi^j}{|\xi|^2}-\frac{\partial H}{\partial x_k}\right).
\end{equation}
Our goal is to show  that the mean curvature $H$ is constant on $M$.
We do this using different  trial covectors $\xi\in T_p^*M$ and substituting them into \eqref{eq:symbolmanip1}.
Let us start, for each fixed $i$, by using $\xi=dx_i$. This leads directly to
\begin{equation}\label{eq:symbolmanip2}
  \sum_{j\neq i}\frac{\partial}{\partial x_i}\II_{jj}=0,\qquad\text{ for each }i=1,2,\cdots,d-1.
\end{equation}
This means, in particular, that 
\begin{equation}
\label{eq:meancurv1}
\frac{\partial}{\partial x_i}H=\frac{\partial}{\partial x_i}\II_{ii}
\end{equation}
Recall that the Codazzi equation for a submanifold $M\subset\Omega$ states that
\[
\left(R(X,Y)Z\right)^\perp=\nabla_X\II(Y,Z)-\nabla_Y\II(X,Z),
\]
where $R$ is the Riemannian curvature tensor of the ambient space and $\perp$ denotes the projection on the normal direction 
(see \cite[formula (4.1.6)]{Taylor}).
Now, because the ambient space is $\R^d$, the curvature vanishes and this simply means that $\nabla_X\II(Y,Z)=\nabla_Y\II(X,Z)$. This has the useful consequence that 
\begin{equation}\label{eq:codaz}
\frac{\partial}{\partial x_i}\II(e_j,e_k)\quad\text{is symmetric in }i, j, k.
\end{equation}
Multiplying~\eqref{eq:symbolmanip1} by $2$ and setting $\xi=dx_i+dx_{j}$, $i \neq j$ (we use here that $d\ge 3$ and hence $\dim M~\ge~2$),  leads to
\[
\begin{split}
  0&=\frac{\partial}{\partial x_j}\left(\II_{ii}+2\II_{ij}+\II_{jj}\right)-
     2\frac{\partial}{\partial x_j}H
     +\frac{\partial}{\partial x_i}\left(\II_{ii}+2\II_{ij}+\II_{jj}\right)-
     2\frac{\partial}{\partial x_i}H\\
   &=
     \frac{\partial}{\partial x_j}\left(\II_{ii}+2\II_{ij}+\II_{jj}\right)-
     2\frac{\partial}{\partial x_j}\II_{jj}
     +\frac{\partial}{\partial x_i}\left(\II_{ii}+2\II_{ij}+\II_{jj}\right)-
     2\frac{\partial}{\partial x_i}\II_{ii}\\
   &=
     \frac{\partial}{\partial x_j}\left(\II_{ii}+2\II_{ij}-\II_{jj}\right)
     +\frac{\partial}{\partial x_i}\left(-\II_{ii}+2\II_{ij}+\II_{jj}\right)\\
   &=
     \frac{\partial}{\partial x_j}\left(3\II_{ii}-\II_{jj}\right)
     +\frac{\partial}{\partial x_i}\left(-\II_{ii}+3\II_{jj}\right).
\end{split}
\]
Here the second  equality follows from \eqref{eq:meancurv1},  and the  last equality is obtained  by applying  \eqref{eq:codaz} to the terms containing $\II_{ij}$ 
and  rearranging them afterwards.

Finally,  set $\xi=dx_i+2dx_j$.  Multiplying~\eqref{eq:symbolmanip1} by $5$ and using  the same argument as above  leads to
\[
\begin{split}
  0&=\frac{\partial}{\partial x_i}\left(\II_{ii}+4\II_{ij}+4\II_{jj}\right)-5\frac{\partial}{\partial x_i}\II_{ii}+
     2\frac{\partial}{\partial x_j}\left(\II_{ii}+4\II_{ij}+4\II_{jj}\right)-10\frac{\partial}{\partial x_j}\II_{jj}\\
   &=
     \frac{\partial}{\partial x_i}\left(-4\II_{ii}+4\II_{ij}+4\II_{jj}\right)+
     \frac{\partial}{\partial x_j}\left(2\II_{ii}+8\II_{ij}-2\II_{jj}\right)\\
   &=
     \frac{\partial}{\partial x_i}\left(-4\II_{ii}+12\II_{jj}\right)+
     \frac{\partial}{\partial x_j}\left(6\II_{ii}-2\II_{jj}\right)
\end{split}
\]
The two previous computations lead to a simple linear system which implies, for $i\neq j$,
\[
\frac{\partial}{\partial x_i}\II_{ii}=3\frac{\partial}{\partial x_i}\II_{jj}.
\]
Substitution in~\eqref{eq:symbolmanip2} leads to
\[
  0=3(d-2)\frac{\partial}{\partial x_i}\II_{ii}.
\]
In particular, for $d\geq 3$, this implies
$\frac{\partial}{\partial x_i}\II_{ii}=0$ for each $i=1, \dots, d-1$.
Using~\eqref{eq:symbolmanip2} again, it follows that
\[
\frac{\partial}{\partial x_i}H=\frac{\partial}{\partial x_i}\II_{ii}=0.
\]
That is, the mean curvature $H$ is constant on the hypersurface $M=\partial\Omega\subset\R^d$, and it then follows from Alexandrov's
``Soap Bubble'' theorem \cite{Ale62} that $M$ must be a sphere.  

To prove the result in the opposite direction, we note that the eigenfunctions of the Dirichlet-to-Neumann map on  a sphere are precisely the spherical harmonics.
Therefore, since $\DtN$ and $\Delta_M$ share an orthogonal basis in $L^2(M)$, the two operators commute.

For the first part of the theorem, notice first that the multiplicity of  $\lambda_1=0$ is equal to the number of connected components of $M$, while the multiplicity of $\sigma_1=0$ is equal to one  since $\Omega$ is connected. Hence,  $M$ is connected.
If $d\geq 3$, since  $\mathcal{D}=\sqrt{\Delta_M}$, it follows from the second part of the theorem that $M$ must be a sphere, say of radius $R>0$. However, it is known that in that case,
\[
\lambda_j=\sigma_j^2+\frac{d-2}{R}\sigma_j.
\]
Hence, we must have $d=2$. Thus,  $\Omega\subset\R^2$ is a bounded simply-connected domain. The length $L$ of its boundary is determined by $\lambda_1=\frac{4\pi}{L^2}$, and it follows that
\[
\sigma_2L=\sqrt{\lambda_2}L=\sqrt{\frac{4\pi^2}{L^2}}L=2\pi.
\]
This is the equality case in the Weinstock inequality: $\Omega$ must be a disk (see \cite{Weinst54, GiPo09}).  This completes the proof of the theorem.

\subsection{Discussion and open problems}  The proof of Theorem \ref{thm:rigid} uses the calculus of pseudodifferential operators. This is the  reason we have assumed that $M$ is a smooth surface. It is quite likely that the result holds for surfaces of lower regularity. One possible approach to this problem is to express the Dirichlet-to-Neumann map using layer potentials. We note that Alexandrov's theorem holds for $C^2$ compact embedded surfaces, and it would be interesting to check whether Theorem \ref{thm:rigid}  is true in this case as well. 

It would be also interesting to understand whether the property of a ball described in Theorem \ref{thm:rigid} is ``stable'', i.e. if $\DtN$ and $\Delta_M$ almost commute (in some sense) then $M$ is close to a sphere. In view of stability results for Alexandrov's theorem \cite{MaPo19} it would be sufficient to show that the mean curvature of $M$ is close to a constant in an appropriate norm.

Recall that the proof of the second part of Theorem \ref{thm:rigid} relies  on the condition $d \ge 3$,  since
for $d=2$  the subprincipal symbol $\beta$ of the Dirichlet-to-Neumann map is identically zero. 
\begin{openproblem} 
For planar domains, is it true that the Dirichlet-to-Neumann map and the boundary Laplacian commute if and only if
the domain is a disk or a rotationally symmetric annulus?
\end{openproblem} 
Finally, let us note that it would be interesting to find a geometric characterisation of Riemannian manifolds with (possibly disconnected) boundary, where the Dirichlet-to-Neumann map and the boundary Laplacian commute. The examples of such manifolds include balls in space forms (see \cite{BiSa14}) and cylinders over closed manifolds (see \cite[Example 1.3.3]{GiPo17}). Note that in this setting the symbol calculus can not possibly yield a complete solution. Indeed, any symbolic computation only captures the information in an arbitrary small neighbourhood of the boundary, whereas the Dirichlet-to-Neumann map depends on the interior of the manifold as well. 
\section{The proofs of H\"ormander's identity and its corollaries}\label{sec:identities}
\subsection{Pohozhaev's and H\"ormander's identities}
Let us start with a useful Pohozhaev-type identity (as discussed in Remark \ref{rem:Pohid}) which has various applications, see \cite[Lemma 20]{CGH18}.   
\begin{theorem}[Generalised Pohozhaev's identity] 
\label{thm:poho}
Let $X$ be a complete smooth Riemannian manifold, and let $\Omega \subset X$ be a bounded domain with a Lipschitz boundary. Let $F$ be a Lipschitz  vector field on $\overline{\Omega}$, let $u\in H^1(M)$, and let $U$ be the unique harmonic extension of $u$ into $\Omega$. Then
\begin{equation}\label{eq:pohozaev}
\int_{M} \myinn{F, \nabla U} \partial_n U\, \de v_M - \frac{1}{2}\int_{M} \abs{\nabla U}^2\myinn{F, n} \, \de v_M 
+\frac{1}{2} \int_{\Omega}\abs{\nabla U}^2 \di F\, \de v_\Omega -\int_{\Omega }\myinn{\nabla_{\nabla U}F, \nabla U} \, \de v_\Omega=0.
\end{equation}
\end{theorem}
\begin{proof}
We follow the argument in \cite{CGH18}.  
Since $\Delta U = \di \nabla U =0$ in $\Omega$, using the standard identities for the divergence of a product and for the gradient of a scalar product, we obtain
\[
\di  \left( \myinn{F, \nabla U}\nabla U\right)=\myinn{\nabla \myinn{F,\nabla U},\nabla U}= 
\myinn{\nabla_{\nabla U}F,\nabla U} +\nabla^2U\left[F,\nabla U \right],
\]
where the last term in the right-hand side is understood as the application of the bilinear form given by the Hessian $\nabla^2 U$ to the vectors $F$ and $\nabla U$
(note that the Hessian is well-defined since $U$ is harmonic).
At the same time,
\[
\di\left(|\nabla U|^2F\right)=2\nabla^2 U\left[F,\nabla U\right]+|\nabla U|^2\di  F.
\]
Subtracting the first equality from the second, we get
\[
\di \left(\myinn{F,\nabla U}\nabla U-\frac{1}{2}|\nabla U|^2F\right)=
\myinn{\nabla_{\nabla U}F,\nabla U}-\frac{1}{2}|\nabla U|^2\di F.
\]
Finally,  we integrate this identity over $\Omega$ and use the divergence theorem, noting that $(\nabla U)|_M\in L^2(M)$   since we have assumed $u=U|_M\in H^1(M)$  (see \cite[Theorem A.5]{CWGLS12}).
This completes the proof of Theorem \ref{thm:poho}.
\end{proof}
The original Pohozhaev's identity \cite[Lemma 2]{Poh65} was proved in a different setting. As was mentioned in Remark \ref{rem:Pohid},  the results of this kind are also often referred to as Rellich's identities, see \cite[Theorem 3.1]{HaSi19}.
\begin{proof}[Proof of Theorem \ref{thm:Hormid}] Setting $F=n$ on $M$ in \eqref{eq:pohozaev}, we obtain
\begin{equation*}
\int_{M}\left(\partial_n U\right)^2\, \de v_M-\frac{1}{2}\int_{M}|\nabla U|^2\, \de v_M= \int_{\Omega }\myinn{\nabla_{\nabla U}F,\nabla U} \, \de v_\Omega
-\frac{1}{2}\int_{\Omega}|\nabla U|^2 \di F\, \de v_\Omega.
\end{equation*}
Note that on  $M$ we have 
\[
|\nabla U|^2=|\nabla_{M}u|^2+\left(\partial_n U\right)^2.
\]
Therefore, with account of $\DtN u=\partial_n U$, 
\[
  (\DtN u, \DtN u)_M-\frac{1}{2}\int_M \left(|\nabla_{M} u|^2+\left(\partial_n U\right)^2\right)\,\de v_M  = \int_{\Omega}\myinn{\nabla_{\nabla U}F,\nabla U}\, \de v_\Omega
-\frac{1}{2}\int_{\Omega }|\nabla U|^2 \di F\, dv_\Omega.
\]
Multiplying by $2$ and re-arranging,  we obtain
\[
\myscal{\DtN u, \DtN u}_M-\myscal{\nabla_{M}u, \nabla_{M}u}_M=
\int_{\Omega}\left(2\myinn{\nabla_{\nabla U}F,\nabla U} \, \de v_\Omega -|\nabla U|^2 \di F\right)\, \de v_\Omega,
\]
which completes the proof of the theorem.
\end{proof}

\begin{cor}[\cite{Hor18}]\label{cor:hormcor1}
There exists a constant $C > 0$ depending only on the geometry of $\Omega$ in an arbitrary small neighbourhood of $M$ such that  for any $u \in H^1(M)$ inequality \eqref{eq:Horcor} 
holds, i.e. \begin{equation*}
\abs{\myscal{\DtN u, \DtN u}_M-(\nabla_{M} u, \nabla_M u)_M} \le C (\DtN u, u)_M.
\end{equation*}
\end{cor}
\begin{proof}  Note  that the integrand on the  right-hand side of \eqref{eq:Hormid} is a quadratic form in $\nabla U$ with bounded coefficients, since the vector field $F$ is Lipschitz continuous, see Remark \ref{rem:Lipcontvf}. Therefore, there exists a constant $C>0$ such that
\[
\abs{\int_{\Omega}\left(2\myinn{\nabla_{\nabla U}F,\nabla U} \, \de v_\Omega -|\nabla U|^2 \di F\right)\, \de v_\Omega} \le C \myscal{\nabla U, \nabla U}_\Omega = C\myscal{\DtN u, u}_M,
\]
where the last equality follows from the divergence theorem. 
Moreover, $C$ can be chosen  only depending on the geometry of $\Omega$ in an arbitrary small neighbourhood of $M$, see Remark \ref{rem:distance}.
This completes the proof of the corollary.
\end{proof}
\subsection{An abstract eigenvalue estimate} Before proceeding to the proof of Theorem \ref{thm:Hormbound}, we state the following abstract result generalising the idea of H\"ormander \cite{Hor18}.
\begin{prop}\label{lem:abstract} Let $\CH$ be a Hilbert space with an inner product $\myscal{\cdot,\cdot}_\CH$. Let $\CA, \CB$ be two non-negative self-adjoint operators in $\CH$ with discrete spectra $\Spec(\CA)=\{\alpha_1\le\alpha_2\le\dots\}$ and $\Spec(\CB)=\{\beta_1\le\beta_2\le\dots\}$ and the corresponding  orthonormal bases of eigenfunctions $\{a_k\}$, $\{b_k\}$. Assume additionally that $a_k\in \Dom(\CB)$ and $b_k\in\Dom(\CA^2)$, $k\in\mathbb{N}$, where the domains are understood in the sense of quadratic forms.  Suppose that for some $C>0$
\begin{equation}\label{eq:abstractcond}
\left|\myscal{\CA u, \CA u}_\CH-\myscal{\CB u, u}_\CH\right|\le C\myscal{\CA u,u}_\CH\qquad \text{for all }u\in D:=\Dom(\CB)\cap \Dom(\CA^2).
\end{equation}
Then 
\begin{equation}\label{eq:abstractestimate}
\left|\alpha_k^2-\beta_k\right| \le C\alpha_k
\end{equation}
and consequently 
\begin{equation}\label{eq:abstractestimateroot}
\left|\alpha_k-\sqrt{\beta_k}\right| \le C
\end{equation}
for all $k\in\mathbb{N}$, with the same constant $C$ as in \eqref{eq:abstractcond}.
\end{prop}

\begin{proof} We note that \eqref{eq:abstractcond} is equivalent to 
\begin{equation}\label{eq:abstractcond1}
\begin{cases}
\myscal{\CB u, u}_\CH\le \myscal{\CA u, \CA u}_\CH+C\myscal{\CA u,u}_\CH,\\
\myscal{\CA u, \CA u}_\CH-C\myscal{\CA u,u}_\CH\le \myscal{\CB u, u}_\CH,
\end{cases}
\end{equation}
and \eqref{eq:abstractestimate} is equivalent to 
\begin{equation}\label{eq:abstractestimate1}
\begin{cases}
\beta_k\le \alpha_k^2+C\alpha_k,\\
\beta_k\ge \alpha_k^2-C\alpha_k.
\end{cases}
\end{equation}

From the variational principle for the eigenvalues of $\CB$ and the first inequality in \eqref{eq:abstractcond1} we have
\begin{equation}\label{eq:varprbeta}
\beta_k\le \sup_{0\ne u\in V_k\subset \Dom(\CB)}\frac{\myscal{\CB u, u}_\CH}{\myscal{u,u}_\CH} \le \sup_{0\ne u\in V_k\subset \Dom(\CB)}\frac{\myscal{\CA u, \CA u}_\CH+C\myscal{\CA u,u}_\CH}{\myscal{u,u}_\CH} 
\end{equation}
for any subspace $V_k$ with $\dim V_k=k$. Take $V_k=\Span\{a_1,\dots,a_k\}$. As for any $u=c_1a_1+\dots+c_ka_k\in V_k$ with $|c_1|^2+\dots+|c_k|^2=1$ we have due to orthogonality
\[
\frac{\myscal{\CA u, \CA u}_\CH+C\myscal{\CA u,u}_\CH}{\myscal{u,u}_\CH} =\sum\limits_{j=1}^k |c_j|^2 (\alpha_j^2+C\alpha_j)\le \alpha_k^2+C\alpha_k,
\]
the first inequality \eqref{eq:abstractestimate1} follows immediately from \eqref{eq:varprbeta}. 

We now prove the second inequality  \eqref{eq:abstractestimate1}. Let $K_0:=\max\{k\in\mathbb{N}: \alpha_k\le C\}$. We note that for $k\le K_0$  the second inequality  \eqref{eq:abstractestimate1}
is automatically satisfied since in this case $\beta_k\ge 0\ge \alpha_k\left(\alpha_k-C\right)$, so we need to consider only $k>K_0$. We re-write the second inequality \eqref{eq:abstractcond1} as
\[
\myscal{\tilde{\CA} u, \tilde{\CA} u}_\CH\le \myscal{\CB u, u}_\CH+\frac{C^2}{4}\myscal{u,u}_\CH,
\] 
where $\tilde{\CA}:=\CA-\frac{C}{2}$. Let $\tilde{\alpha}_k^2$ denote the eigenvalues of ${\tilde{\CA}\,}^2$ enumerated in non-decreasing order. We note that $\tilde{\alpha}_k^2=\left(\alpha_k-\frac{C}{2}\right)^2$ for $k>K_0$ (this may not be the case for $k\le K_0$ but as mentioned above we can ignore these values of $k$). Writing down the variational principle for  $\tilde{\alpha}_k^2$ similarly to  \eqref{eq:varprbeta} and choosing a test subspace $V_k=\Span\{b_1,\dots,b_k\}$ leads in a similar manner to 
\[
\tilde{\alpha}_k^2=\left(\alpha_k-\frac{C}{2}\right)^2\le \beta_k + \frac{C^2}{4},
\]
which gives the second inequality \eqref{eq:abstractestimate1} after a simplification.

Finally, we note that \eqref{eq:abstractestimate} implies, for $\alpha_k \beta_k\ne 0$, 
\[
\left|\alpha_k-\sqrt{\beta_k}\right| \le C\frac{\alpha_k}{\alpha_k+\sqrt{\beta_k}}\le C,
\]
yielding \eqref{eq:abstractestimateroot}. Note that $\alpha_k=0$ implies $\beta_k=0$ by \eqref{eq:abstractestimate1}.
\end{proof} 

\subsection{Proof of Theorem \ref{thm:Hormbound}} We apply Lemma \ref{lem:abstract} with $\mathcal{A}=\DtN$, $\mathcal{B}=\Delta_M$, and therefore $\alpha_k=\sigma_k$ and $\beta_k=\lambda_k$, taking into account Corollary \ref{cor:hormcor1}, the obvious $\myscal{\Delta_M u,u}_M=\myscal{\grad_Mu,  \grad_Mu}_M$, and the fact that, as follows from \cite[Proposition 7.4]{MiTa99}, the eigenfunctions of $\DtN$ belong to $H^1(M)$; note that since $M$ is not smooth, this is not \emph{a priori}  evident. Thus we can take $D=H^1(M)$, and \eqref{eq:steklap} follows immediately from \eqref{eq:abstractestimateroot}.
\subsection{Applications to spectral asymptotics} 
\label{subsec:asymp}
Theorem \ref{thm:Hormbound} allows us to prove results on the asymptotic distribution of Steklov eigenvalues using similar results for the Laplacian.
\begin{proof}[Proof of Theorem \ref{thm:rough1}] Since the boundary of $\Omega$ has regularity $C^{2,\alpha}$ for some $\alpha>0$, the normal vector
to the boundary has regularity $C^{1,\alpha}$. Indeed, the boundary is locally given  by a graph of a $C^{2,\alpha}$ function, and the normal vector is calculated in terms of its first derivatives.  Hence the induced Riemannian metric on $M=\partial \Omega$ has $C^{1,\alpha}$ coefficients. 
At the same time, it was shown in \cite[Theorem 3.1]{Iv00} (see also \cite{Ziel99} for a similar result under slightly stronger regularity assumptions)
that sharp Weyl's law holds for the Laplace eigenvalues on manifolds with  a Riemannian metric  having coefficients of regularity 
$C^{1,\alpha}$ for some $\alpha>0$.  In other words, the asymptotic formulas \eqref{eq:weylawcount}  and \eqref{eq:weylaw} hold  on $M$ 
with $\sigma_k$ replaced by $\sqrt{\lambda_k}$.
Therefore, in view of  \eqref{eq:steklap}  they  holds for $\sigma_k$ as well, and this completes the proof of the theorem.
\end{proof}
\begin{remark} To our knowledge, in dimension $d>2$, the sharp asymptotic formula \eqref{eq:weylaw} was previously available in the literature 
only for domains with smooth boundaries.
\end{remark}
\begin{proof}[Proof of Theorem \ref{thm:rough2}] Since $\Omega$ is two-dimensional, its boundary $M$  has dimension one and  is therefore 
isometric to a union of circles.
Hence, the Laplace eigenvalues of $M$ are explicitly known (recall that the unit circle has the Laplace spectrum given by $0, 1, 1, 4, 4, \dots, k^2, k^2, \dots$) and satisfy the sharp Weyl's law. Therefore, by  \eqref{eq:steklap} the sharp Weyl's law \eqref{eq:weylaw} holds for 
the Steklov eigenvalues $\sigma_k$ which yields \eqref{eq:asymptwo} since $d=2$. This completes the proof of the Theorem.
\end{proof}
 \begin{remark}
One expects sharper results to hold for domains in two dimensions. In particular, for domains with $C^r$ boundaries, $r \ge 1$, it is likely that analogue of 
\eqref{eq:gpps} holds, with the right-hand side decaying polynomially in $k$, with the order of decay depending on $r$. Some results in this direction have been obtained in \cite{CaLa21}.
\end{remark} 
\begin{remark}\label{rem:bdryLap} It would be interesting to understand how much one can relax the $C^{1,1}$ regularity assumption so that the  
asymptotic formula \eqref{eq:asymptwo} remains true. For instance, it is known to hold for planar curvilinear polygons with sides that are $C^5$ regular,
 see \cite{LPPS19}.  Moreover, for a large class of curvilinear polygons, $\abs{\sigma_k - \sqrt{\lambda_k}}=o(1)$, 
provided the boundary Laplacian is defined as a certain quantum graph Laplacian on the circular graph modelled by the boundary, with the matching conditions at the vertices determined by the corresponding angles. 
\end{remark}
We say that the eigenvalue asymptotics satisfies a {\it rough} Weyl's law if formula \eqref{eq:weylaw} holds with the error term
$o\left(\sigma^{d-1}\right)$ instead of $O\left(\sigma^{d-2}\right)$.  For Euclidean domains with $C^2$ boundary, a rough Weyl's law for Steklov eigenvalues was first obtained by L. Sandgren in  in \cite{Sand55}. Using heavier machinery, a similar  result can be also proved for   Euclidean domains with piecewise $C^1$ boundaries \cite{Agr06}.  Let us conclude this section by a challenging open problem going back to M. Agranovich in 2000s.
\begin{openproblem}
Show that a rough Weyl's law  holds for the Steklov problem on any bounded Lipschitz domain in a smooth Riemannian manifold. 
\end{openproblem}
\section{Dirichlet-to-Neumann map for the Helmholtz equation}\label{sec:lambda}
\subsection{Parameter-dependent Dirichlet-to-Neumann map}
Let, as before, $\Omega$ be a bounded domain in a complete smooth Riemannian manifold $X$ of dimension $d\ge 2$; we assume that the boundary $M=\partial\Omega$ is Lipschitz. Consider the standard Dirichlet and Neumann Laplacians $\Delta^\Dir=\Delta^\Dir_\Omega$ and $\Delta^\Neu=\Delta^\Neu_\Omega$ acting on $\Omega$. Their spectra are discrete, and we will denote their eigenvalues by
\[
0<\mu_1^\Dir<\mu_2^\Dir\le\dots  \nearrow +\infty,
\]
and 
\[
0=\mu_1^\Neu<\mu_2^\Neu\le\dots  \nearrow +\infty,
\]
respectively.

Let $\mu\in\mathbb{R}\setminus\Spec(\Delta^\Dir_\Omega)$. Then the boundary value problem 
\begin{equation}\label{eq:Helm}
\begin{cases}
\Delta U= \mu U\qquad&\text{in }\Omega,\\
U=u\qquad&\text{on }M
\end{cases}
\end{equation}
has a unique solution $U:=U_\mu:=U_{\mu,u}\in H^1(\Omega)$ for every $u\in H^{1/2}(M)$. We will call $U$ the $\mu$-\emph{Helmholtz extension} of $u$; for $\mu=0$ it is just a harmonic extension. The parameter-dependent operator
\[
\mathcal{D}_\mu:  
\begin{aligned}
H^{1/2}(M)&\to H^{-1/2}(M),\\
u &\mapsto  \partial_n U_\mu,
\end{aligned}
\]
is called the \emph{Dirichlet-to-Neumann map} for the Helmholtz equation $\Delta U= \mu U$. The Dirichlet-to-Neumann map for the Laplacian, $\mathcal{D}$, considered in \S\S\ref{sec:intro}--\ref{sec:identities}, is just the special case $\mathcal{D}=\mathcal{D}_0$. The DtN map $\mathcal{D}_\mu$ can be also defined for $\mu$ coinciding with an eigenvalue $\mu_k^\Dir$ of the Dirichlet Laplacian if its domain is restricted to the orthogonal complement, in $L^2(M)$, to the span of the normal derivatives of the corresponding Dirichlet eigenfunctions.

For every $\mu\in\mathbb{R}\setminus\Spec\left(\Delta^\Dir\right)$, the  DtN map $\mathcal{D}_\mu$ is a self-adjoint operator in $L^2(M)$ with a discrete spectrum; we enumerate its eigenvalues with account of multiplicities as 
$\sigma_{\mu,1}<\sigma_{\mu,2}\le\dots \nearrow +\infty$. By the variational principle and integration by parts,
\begin{equation}\label{eq:varprsigmamu}
\begin{split}
\sigma_{\mu,k}&=\inf_{\substack{V_k \subset H^{1/2}(M)\\\dim V_k=k}}\ \sup_{0 \neq u \in V_k} \frac{\myscal{\DtN_\mu u, u}_M}{\myscal{u,u}_M} \\
&=\inf_{\substack{W_k \subset \{U\in H^{1}(\Omega):\ \Delta U=\mu U\}\\\dim W_k=k}}\ \sup_{0 \neq U \in W_k} \frac{\myscal{\nabla U, \nabla U}_\Omega-\mu\myscal{U,U}_\Omega}{\myscal{U,U}_M}
\end{split}
\end{equation}
(for $\mu<\mu_1^\Dir$ one can take $W_k\subset H^1(\Omega)$ in the second $\inf\sup$).
Moreover, if $M$ is smooth, then $\DtN_\mu$ is an elliptic pseudodifferential operator of order one with the same principal symbol as $\DtN=\DtN_0$ and therefore as $\sqrt{\Delta_M}$, with the same eigenvalue asymptotics \eqref{eq:weylawcount}, \eqref{eq:weylaw}.

In the remarkable paper \cite{Fri91}, L. Friedlander investigated the dependence of the eigenvalues of operator $\DtN_\mu$ upon the parameter $\mu$ in the Euclidean setting, and used them to prove the inequalities
\begin{equation}\label{eq:Fried}
\mu^\Neu_{k+1}\le \mu^\Dir_k, \qquad k\in\mathbb{N},
\end{equation}
between the  Neumann and Dirichlet eigenvalues for any bounded domain $\Omega\subset\R^d$ with smooth boundary $M$ (this was later extended to non-smooth boundaries by N. Filonov \cite{Fil05} using a different approach). Friedlander's results were based on the following main observations:
\begin{itemize} 
\item The eigenvalues $\sigma_{\mu,k}$ are monotone decreasing continuous functions of $\mu$ on each interval of the real line not containing points of $\Spec(\Delta^\Dir)$.
\item At each Neumann eigenvalue $\mu^\Neu\in \Spec(\Delta^\Neu)$ of multiplicity $m_{\mu^\Neu}$, exactly $m_{\mu^\Neu}$ eigenvalue curves $\sigma_{\mu, k}$ (as functions of $\mu$) cross the axis $\sigma=0$ from the upper half-plane into the lower one.
\item At each Dirichlet eigenvalue $\mu^\Dir\in \Spec(\Delta^\Dir)$ of multiplicity $m_{\mu^\Dir}$, exactly $m_{\mu^\Dir}$ eigenvalue curves $\sigma_{\mu, k}$ (as functions of $\mu$) blow down to $-\infty$ as $\mu\to \mu^\Dir-0$ and blow up to $+\infty$ as $\mu\to \mu^\Dir+0$.
\item Therefore the eigenvalue counting functions of the Dirichlet problem $N^\Dir(\mu):=\#\{\mu_k^\Dir<\mu\}$, of the Neumann problem $N^\Neu(\mu):=\#\{\mu_k^\Neu<\mu\}$, and of the DtN operator $N^{\DtN_\mu}(\sigma):=\#\{\sigma_{\mu, k}<\sigma\}$, are related, for any $\mu\in\mathbb{R}$, by the relation
\[
N^\Neu(\mu)-N^\Dir(\mu)=N^{\DtN_\mu}(0).
\]
\end{itemize}
Friedlander then demonstrated that $N^{\DtN_\mu}(0)$, that is, the number of negative eigenvalues of $\DtN_\mu$, is at least one for any $\mu>0$ for a domain $\Omega$ in a Euclidean space, thus implying \eqref{eq:Fried} (this need not be true for domains on a Riemannian manifold, see \cite{Maz91}). We also refer to \cite{ArMa12} for extensions of Friedlander's approach to a Lipschitz case and a comprehensive discussion of various other generalisations and alternative  approaches, and to \cite{Saf08} for an abstract scheme encompassing the above. 

A typical behaviour of eigenvalues $\sigma_{\mu,k}$ as functions of $\mu$ is illustrated by Figure \ref{fig:fig1}, which shows some of the eigenvalues for a unit disk for which the spectrum of the Dirichlet-to-Neumann map $\DtN_\mu$ is given by the multisets 
\begin{equation}\label{eq:DtNdiskev}
\Spec(\DtN_\mu)=\begin{cases}
\left\{\frac{\sqrt{\mu}J'_n\left(\sqrt{\mu}\right)}{J_n\left(\sqrt{\mu}\right)}, n\in\mathbb{N}\cup\{0\}\right\}\qquad&\text{if }\mu\ge 0,\\[2ex]
\left\{\frac{\sqrt{-\mu}I'_n\left(\sqrt{-\mu}\right)}{I_n\left(\sqrt{-\mu}\right)}, n\in\mathbb{N}\cup\{0\}\right\}\qquad&\text{if }\mu<0,
\end{cases}
\end{equation}
with $J_n$ and $I_n$ being the Bessel functions and the modified Bessel functions, respectively, and eigenvalues with $n>0$ should be taken with multiplicity two.

\begin{figure}[htb!]
\begin{center}
\includegraphics[width=0.75\textwidth]{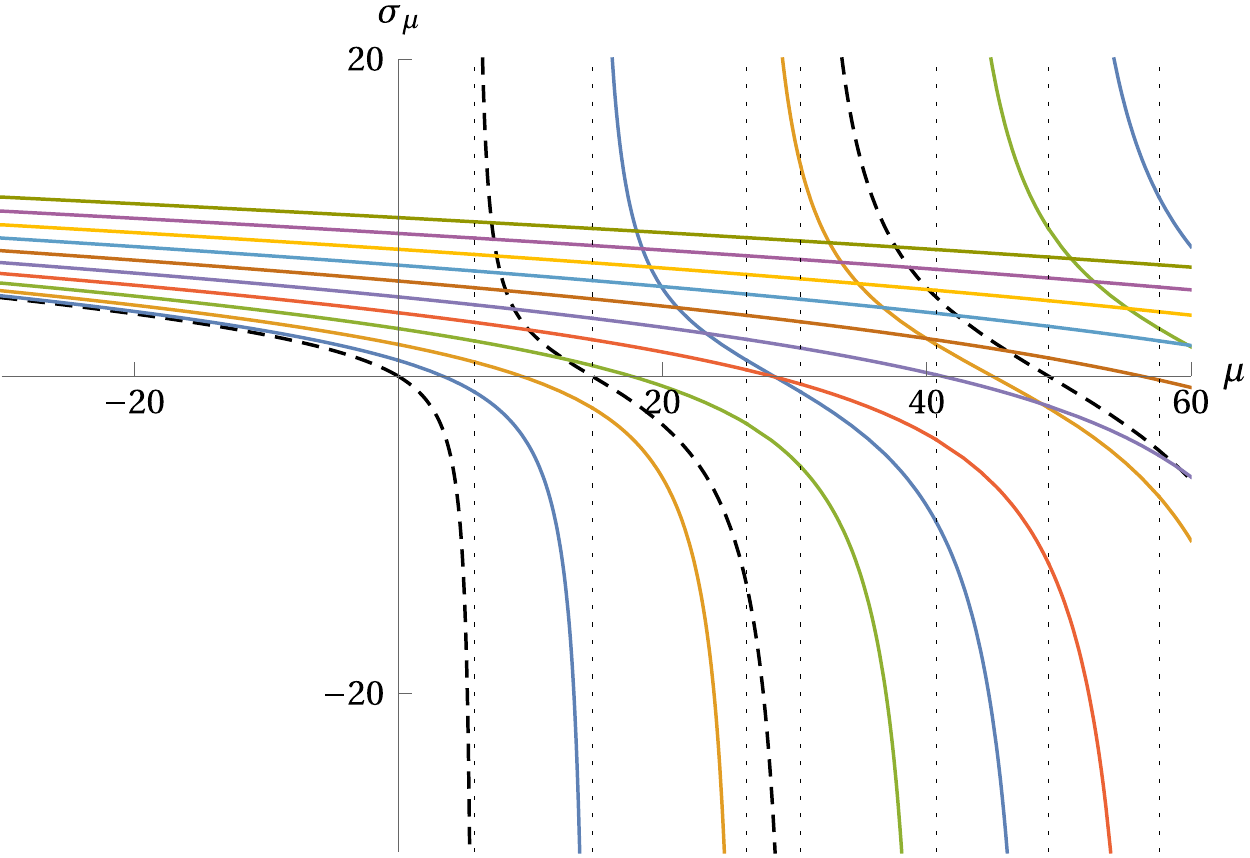}
\caption{Some eigenvalues $\sigma_{\mu,k}$ of $\DtN_\mu$ for a unit disk plotted as functions of $\mu$. The solid curves correspond to $n>0$ in \eqref{eq:DtNdiskev} and are double; the dashed curve corresponds to $n=0$ and is single. The dotted vertical lines indicate the positions of the Dirichlet eigenvalues (points from $\Spec(\Delta^\Dir)$), and the intersections of the curves with the axis $\sigma=0$ are at the Neumann eigenvalues from $\Spec(\Delta^\Neu)$. \label{fig:fig1}}
\end{center}
\end{figure}

We will also make use of the following generalised Pohozhaev's identity which extends Theorem \ref{thm:poho} to solutions of the Helmholtz equation. 
\begin{theorem}[Generalised Pohozhaev's identity {\cite[Theorem 3.1]{HaSi19}}] \label{thm:pohogen}
Let $X$ be a complete smooth Riemannian manifold, and let $\Omega \subset X$ be a bounded domain with a Lipschitz boundary. Let $F$ be a Lipschitz  vector field on $\overline{\Omega}$, let $u\in H^1(M)$, and let $U=U_{\mu, u}$ be the $\mu$-Helmholtz extension of $u$ into $\Omega$. Then
\begin{equation}\label{eq:pohozaevgen}
\begin{split}
\int_{M} \myinn{F,\nabla U}\partial_n U \, \de v_M - \frac{1}{2}\int_{M} \abs{\nabla U}^2 \myinn{F,n}\, \de v_M + \frac{\mu}{2} \int_{M} u^2\myinn{F,n} \, \de v_M&\\
+\frac{1}{2} \int_{\Omega}\abs{\nabla U}^2 \di F\, \de v_\Omega -\int_{\Omega} \myinn{\nabla_{\nabla U}F,\nabla U} \, \de v_\Omega - \frac{\mu}{2} \int_{\Omega} U^2 \di F\, \de v_\Omega &=0.
\end{split}
\end{equation}
\end{theorem}
\subsection{The case $\mu\le 0$}  
We aim to prove the following
\begin{theorem}\label{thm:Hormboundgen}
Let $X$ be a complete smooth Riemannian manifold, and let $\Omega \subset X$ be a  bounded domain with a $C^{1,1}$ boundary. Let $\mu\le 0$. Then, with some constant $C>0$, the bounds
\begin{equation}
\label{eq:steklapgen}
\left|\sigma_{\mu,k} - \sqrt{\lambda_k-\mu}\right| < C.
\end{equation}
hold uniformly over  $\mu\in(-\infty,0]$ and $k\in\mathbb{N}$.
\end{theorem}

Before proceeding to the actual proof of Theorem \ref{thm:Hormboundgen} we require a Helmholtz  analogue of Theorem \ref{thm:Hormid}. Repeating literally the proof of Theorem \ref{thm:Hormid}, with account of extra $\mu$-dependent terms in  \eqref{eq:pohozaevgen} compared to \eqref{eq:pohozaev}, we arrive at
\begin{theorem}\label{thm:Hormidgen}
Let  $X$ be a complete smooth Riemannian manifold, and let $\Omega \subset X$ be a  bounded domain with a $C^{1,1}$ boundary.
Let  $F$  be a Lipschitz vector field on $\overline{\Omega}$ such that $F|_M=n$, let $u\in H^1(M)$, and let $U=U_{\mu, u}$ be the $\mu$-Helmholtz extension of $u$ into $\Omega$..  Then
\begin{equation}\label{eq:Hormid1}
\myscal{\DtN_\mu u,\DtN_\mu u}_M-\myscal{\nabla_{M} u,\nabla_M u}_M+\mu\myscal{u,u}_M=\int_{\Omega}\left(2\myinn{\nabla_{\grad U}F,\grad U}-|\grad U|^2 \di F + \mu U^2 \di F\right)\, \de v_\Omega.
\end{equation}
\end{theorem}

Theorem \ref{thm:Hormidgen} leads to the crucial Helmholtz analogue of Corollary \ref{cor:hormcor1}:
\begin{cor}\label{cor:hormcorgen}
Under conditions of Theorem \ref{thm:Hormidgen},  there exists a constant $C > 0$ depending only on the geometry of $\Omega$ in an arbitrary small neighbourhood of $M$ such that  for any $u \in H^1(M)$ and any $\mu\le 0$
\begin{equation}\label{eq:hormineqmod}
\abs{\myscal{\DtN_\mu u, \DtN_\mu u}_M-\myscal{(\Delta_M-\mu) u,u}_M} \le C \myscal{\DtN_\mu u, u}_M.
\end{equation}
\end{cor}
 
\begin{proof}[Proof of Corollary  \ref{cor:hormcorgen}] Take the absolute values in both sides of equality \eqref{eq:Hormid1}. Then the left-hand side becomes the left-hand side of \eqref{eq:hormineqmod} after an integration by parts on $M$. The first two terms in the right-hand side can be estimated above by $C\myscal{\nabla U, \nabla U}_\Omega$ by the same argument as in Corollary \ref{cor:hormcor1}, and the last term by $C|\mu|\myscal{U, U}_\Omega$ (possibly with a different constant but also depending on $F$ only).
Since for non-positive $\mu$ we have $|\mu|=-\mu$, the bound in the right-hand side becomes 
\[
C\left(\myscal{\nabla U, \nabla U}_\Omega-\mu \myscal{U, U}_\Omega\right)=C\myscal{\DtN_\mu u, u}_M,
\]
and the result follows. 
\end{proof} 

Theorem \ref{thm:Hormboundgen} now follows immediately from Corollary  \ref{cor:hormcorgen} and Lemma \ref{lem:abstract} by taking in the latter $\CA=\DtN_\mu$ and $\CB=\Delta_M-\mu$ (which are both non-negative for $\mu\le 0$).  

\begin{remark} The remarkable feature of Theorem \ref{thm:Hormboundgen} is that the constant appearing in the right-hand side of the bound is in fact independent of both the eigenvalue's number $k$ and the parameter $\mu$ (as long as $\mu$ is non-positive). As we will see shortly, such uniform bounds are impossible if the boundary $M$ has corners.
\end{remark}

We illustrate  Theorem \ref{thm:Hormboundgen} by plotting, in Figure \ref{fig:fig2}, some eigenvalues of $\DtN_\mu$ for a unit disk and, for comparison, the values of $\sqrt{\lambda_k-\mu}$, in two regimes: firstly, for negative $\mu$ close to zero, and  low eigenvalues of $\DtN_\mu$, and secondly, for very large negative $\mu$, and relatively high eigenvalues of $\DtN_\mu$.

\begin{figure}[htb!]
\begin{center}
\includegraphics[width=0.49\textwidth]{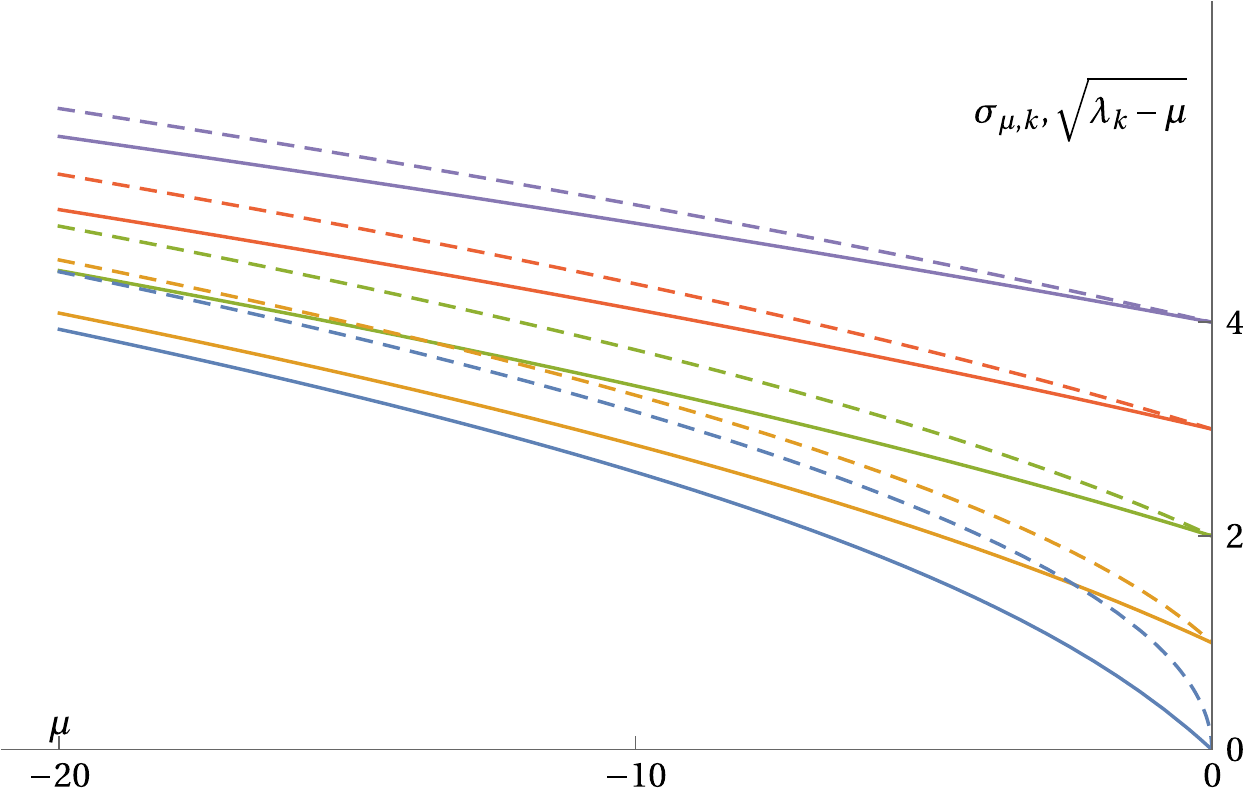}\hfill \includegraphics[width=0.49\textwidth]{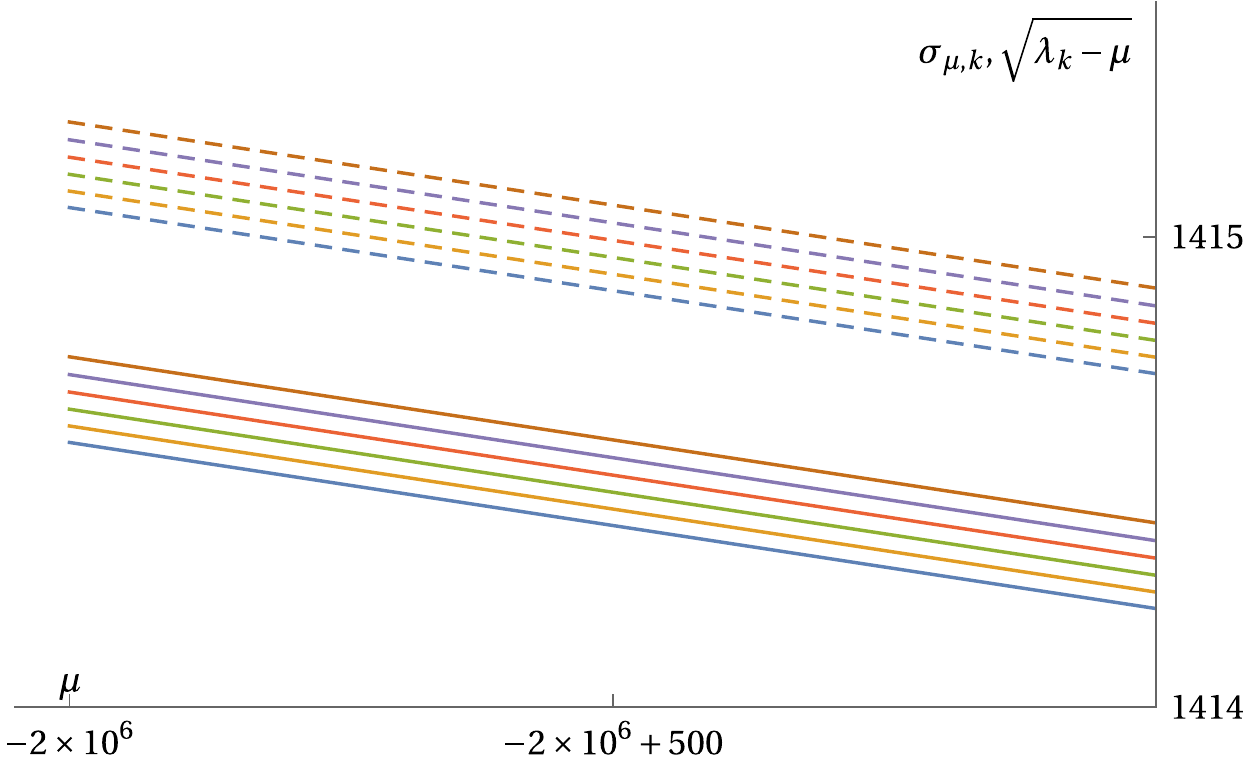}
\caption{Some eigenvalues $\sigma_{\mu,k}$ of $\DtN_\mu$ for a unit disk plotted as functions of $\mu$ (solid curves), and, for comparison, the plots of $\sqrt{\lambda_k-\mu}$ (dashed curves). In the left figure, $\mu\in[-20,0]$, and $k$ is chosen in the set $\{1,3,5,7,9\}$. In the right figure, $\mu\in[-2\times 10^6,  -2\times 10^6+10^3]$, and $k$ is chosen in the set $\{100,102,104,106,108\}$.\label{fig:fig2}}
\end{center}
\end{figure}
\subsection{DtN--Robin duality and domains with corners} Consider the two-parametric problem
\begin{equation}\label{eq:twoparam}
\begin{cases}
\Delta v= \mu v\qquad&\text{in }\Omega,\\
\partial_n v=\sigma v\qquad&\text{on }M.
\end{cases}
\end{equation}
There are two ways in which we can treat \eqref{eq:twoparam} as a spectral problem. Firstly, as we have already done before, we can treat $\sigma$ as a spectral parameter, and $\mu\in\mathbb{R}$ as a given parameter; then for every $k\in\mathbb{N}$ the eigenvalues $\sigma_{\mu,k}$ (that is, the values of $\sigma$ for which there exists a non-trivial solution $v\in H^1(\Omega)$ of \eqref{eq:twoparam}) are exactly the eigenvalues of the Dirichlet-to-Neumann map $\DtN_\mu$. Conversely, we can treat $\mu$ as a spectral parameter and $\sigma\in\mathbb{R}$ as a given parameter. The corresponding eigenvalues $\mu_{-\sigma,k}^{\Rob}$ are then exactly the eigenvalues of the \emph{Robin Laplacian} $\Delta^{\Rob,\gamma}$ with $\gamma=-\sigma$, that is of the Laplacian in $\Omega$ subject to the boundary condition
\[
\left(\partial_n +\gamma\right)v=0\qquad\text{on }M,
\]
with the quadratic form 
\[
\myscal{\Delta^{\Rob,\gamma}v, v}_\Omega=\myscal{\nabla v, \nabla v}_\Omega+\gamma\myscal{v,v}_M,\qquad v\in H^1(\Omega).
\]
(Note that there is no uniform convention on the choice of sign in the Robin condition, therefore some care should be exercised when comparing results in the literature.) It is immediately clear that $\mu\in\Spec\left(\Delta^{\Rob,-\sigma}\right)$ if and only if $\sigma\in\Spec\left(\DtN_\mu\right)$, and it is easy to check that in this case the dimensions of the corresponding eigenspaces of  $\Delta^{\Rob,-\sigma}$ and $\DtN_\mu$ coincide, see  \cite{ArMa12}.  Moreover, due to monotonicity of eigenvalues $\sigma_{\mu, k}$ of $\DtN_\mu$ in the parameter $\mu\in(-\infty, \mu^\Dir_1)$, the functions $\mu\mapsto \sigma_{\mu, k}$ on this interval are just the inverse functions of $\sigma\mapsto \mu_{-\sigma,k}^{\Rob}$. 

The study of the Robin eigenvalues $\mu_{\gamma, k}^{\Rob}$, in particular in the physically important regime $\gamma\to-\infty$, has grown significantly in the last two decades, starting with some acute observations in \cite{LOS98}, their rigorous justification in \cite{LePa08}, and most recently mostly due to M. Khalile, K. Pankrashkin, N. Popoff and collaborators, see in particular \cite{Kha18, KhPa18, KO-BP18, Pan20, Pop20} and references therein. Without going into the full details of these works, we mention only that, as it turns out, the asymptotics of the Robin eigenvalues in $\Omega$ as $\gamma\to-\infty$ depends dramatically on the smoothness of $M=\partial\Omega$. Specifically, as shown in \cite{Kha18}, if $\Omega$ is a curvilinear polygon in $\mathbb{R}^2$ with at  least one angle less than $\pi$, then the following dichotomy is observed. There exists a number $K\ge 1$ such that, as $\gamma\to-\infty$, the first $K$ Robin eigenvalues have the asymptotics
\begin{equation}\label{eq:Robinasympt}
\mu_{\gamma,k}^{\Rob}=
-C_k \gamma^2+o\left(\gamma^2\right)\quad\text{for }k=1,\dots K,\text{ with }C_1\ge \dots\ge C_K>1,
\end{equation}
whereas all the remaining eigenvalues behave as
\[
\mu_{\gamma,k}^{\Rob}=
-\gamma^2+o\left(\gamma^2\right)\quad\text{for }k>K.
\]
Here the number $K$ and the constants $C_1$, \dots, $C_K$ are determined by the angles at the corners of $\Omega$; in some cases this dependence can be made explicit. For a smooth boundary one should take $K=0$; the remainder estimates in the formulae above admit various improvements. 

Based on \eqref{eq:Robinasympt}, we make the following observation showing that our uniform bounds of Theorem \ref{thm:Hormboundgen} \emph{cannot} be extended in the same uniform manner to domains with corners, whatever boundary Laplacian we choose (see Remark \ref{rem:bdryLap}). 

\begin{prop}
\label{prop:polyg}
 Let $\Omega$ be a curvilinear polygon in $\mathbb{R}^2$ with at least one angle less than $\pi$. Then the bounds \eqref{eq:steklapgen} cannot hold uniformly over  $\mu\in(-\infty,0]$ and $k\in\mathbb{N}$ for any choice of a boundary Laplacian $\Delta_M$ with eigenvalues $\lambda_k$. 
\end{prop}

\begin{proof} Suppose, for contradiction, that the bounds \eqref{eq:steklapgen} hold uniformly. Passing, for a fixed $k$, in \eqref{eq:steklapgen}  to the asymptotics as $\mu\to -\infty$ with account of 
\[
\sqrt{\lambda_k-\mu}=\sqrt{\vphantom{\lambda_k}-\mu}+O\left((-\mu)^{-\frac{1}{2}}\right),
\]      
and using the DtN--Robin duality, we deduce that all the Robin eigenvalues should then satisfy
\[
\mu_{\gamma,k}^{\Rob}=-\gamma^2+o\left(\gamma^2\right)\qquad\text{as }\gamma\to-\infty,
\]
thus contradicting the condition $C_1, \dots, C_K>1$ in \eqref{eq:Robinasympt}.
\end{proof}
\subsection{The case $\mu>0$} For simplicity, in this subsection we assume that $\Omega\subset\mathbb{R}^d$ with a smooth boundary $M$.  It is immediately clear, for example from Figure  \ref{fig:fig1}, that there is little hope of extending the simple bound \eqref{eq:steklapgen} of Theorem \ref{thm:Hormboundgen} to the case $\mu>0$ as the first $m$ eigenvalues of the DtN map blow up to $-\infty$ as $\mu$ approaches the Dirichlet eigenvalues $\mu^\Dir$ of multiplicity $m$ from below. Indeed, using the results of \cite{Fil15} or \cite{BBBT18} for the asymptotics of low eigenvalues  of the Robin problem with parameter $\gamma\to+\infty$, and the Robin--DtN duality, one can easily see that as $\mu\to\mu^\Dir-0$, the first $m$ eigenvalues of $\DtN_\mu$ behave asymptotically as
\[
\sigma_{\mu,k}=O\left(\frac{1}{\mu^\Dir-\mu}\right),\qquad k=1,\dots, m.
\]
Therefore, any conceivable generalisation of Theorem \ref{thm:Hormboundgen} to the case $\mu>0$  should take into account the distance 
\[
d^\Dir(\mu):=\operatorname{dist}\left(\mu,\Spec\left(\Delta^\Dir\right)\right)
\]
between $\mu$ and the Dirichlet spectrum.

Taking Theorem \ref{thm:Hormidgen} as a starting point, we in fact obtain 
\begin{theorem}\label{thm:mupos}  
Let $\Omega\subset\mathbb{R}^d$ with a smooth boundary $M$.  Let $\mu_0>0$.  Then there exist positive constants $C$ and $C_1$ depending only on the geometry of $\Omega$ and a positive constant $C_2=C_2(\mu_0)$ which additionally depends on $\mu_0$ such that
\begin{equation}\label{eq:mupos}
-C\sigma_{\mu,k}-\frac{C_2}{\left(d^\Dir(\mu)\right)^2}\le \lambda_k-\sigma_{\mu,k}^2\le C\sigma_{\mu,k}+C_1\mu\left(1+\frac{\mu^2}{\left(d^\Dir(\mu)\right)^2}\right)\quad\text{for all }k\in\mathbb{N}.
\end{equation}
The first inequality in \eqref{eq:mupos} holds uniformly over all $\mu\in[0,\mu_0]$, and the second one uniformly over all $\mu\ge 0$. 
\end{theorem}

We outline a sketch of the proof of Theorem \ref{thm:mupos}. The main problem is that for $\mu\ge 0$ we can no longer deduce an analogue of Corollary \ref{cor:hormcorgen} from \eqref{eq:Hormid1} and then apply Proposition \ref{lem:abstract}, for a couple of reasons: firstly, because the operators $\DtN_\mu$ and $\Delta_M-\mu$ are no longer non-negative, and secondly, because for positive $\mu$ the bound on the right-hand side of \eqref{eq:Hormid1} becomes 
\[
C\left(\myscal{\nabla U, \nabla U}_\Omega+\mu \myscal{U, U}_\Omega\right)=C\left(\myscal{\DtN_\mu u, u}_M+2\mu \myscal{U, U}_\Omega\right),
\]
introducing an extra $\mu$-dependent term. 

To bypass these difficulties, we first write the $\mu$-Helmholtz extension $U_{\mu,u}$ of $u\in H^{\frac{1}{2}}(M)$ (that is, the solution of \eqref{eq:Helm})  in terms of the harmonic extension $U_{0,u}$ of $u$ and the resolvent of the Dirichlet Laplacian as
\[
U_{\mu, u} = \left(1+\mu\left(\Delta^\Dir-\mu\right)^{-1}\right)U_{0,u}.
\]
We further use the standard resolvent norm bound in terms of the distance to the spectrum,
\[
\left\|\left(\Delta^\Dir-\mu\right)^{-1}U\right\|_\Omega\le \frac{1}{d^D(\mu)}\|U\|_\Omega,
\]
together with the bound
\[
\|U_{0,u}\|_\Omega \le \operatorname{const}\cdot \|u\|_M
\]
(see e.g.\ \cite[Corollary 5.5]{JeKe95}) and a bound on the first Robin eigenvalue \cite[formula (1.7)]{Fil15} which after using the DtN--Robin duality becomes
\[
\sigma_{\mu,1} \ge -  \operatorname{const}\cdot\max\left\{1, \frac{1}{d_+^\Dir(\mu)}\right\}
\]
with some constant  independent of $\mu$ and with
\[
d_+^\Dir(\mu):=\operatorname{dist}\left(\mu, \Spec(\Delta^\Dir)\cap[\mu, +\infty)\right)\ge d^\Dir(\mu).
\]
Theorem  \ref{thm:mupos} then follows by using an extended version of Proposition \ref{lem:abstract} which takes care of the extra terms. We leave out the details.
\section{Dirichlet-to-Neumann operator on forms and the boundary Hodge Laplacian}\label{sec:forms}
\subsection{Notation}
 Given a $m$-dimensional Riemannian manifold $Y$ with or without boundary (in our case either $Y=\Omega$ or $Y=M$),  
we denote by $\Lambda^p(Y)$  the space of smooth differential $p$-forms on $Y$, $0 \le p \le m$. 
Throughout  this section $\Omega$ is a compact smooth Riemannian manifold with boundary, $\partial \Omega = M$. 
We assume in addition that $\Omega$ is orientable, so that we can use the standard Hodge theory.

Denote by $\mathfrak{i}:\Omega \to M$ the embedding
of the boundary. Given a $p$-form $\omega \in \Lambda^p(\Omega)$, the $p$-form $\mathfrak{i}^*\omega \in \Lambda^p(M)$ is a part of the 
Dirichlet data for $\omega$.  Let  $d:\Lambda^p(\Omega) 
\to \Lambda^{p+1} (\Omega)$ be the {\em differential}, $i$ be the {\em interior product} and $n$, as before, be the outward normal vector to the boundary.
Then $i_n d\omega~\in~\Lambda^p(M)$ is a part of the Neumann data of $\omega$ (in a slight abuse of notation, $d\omega$ is understood here as the restriction of this form to the boundary).

Let $\delta: \Lambda^p(Y) \to \Lambda^{p-1}(Y)$  be the {\it codifferential} on the space of $p$-forms and $\star: \Lambda^p(Y) \to \Lambda^{m-p}(Y)$  be the {\em Hodge star} operator. We recall the standard relations 
$\star \star = (-1)^{p(m-p)}$ and $\star \delta = (-1)^p d \star$, 
The operator $\Delta=d\delta +\delta d: \Lambda^p(Y) \to \Lambda^p(Y)$ is the  {\it Hodge Laplacian}.

Finally, recall that a form $\omega \in \Lambda^p(Y)$ is called {\it closed} if $d\omega=0$, {\it co-closed} if $\delta \omega =0$, {\it exact} if $\omega=d\alpha$,
$\alpha \in \Lambda^{p-1}(Y)$ and {\it harmonic} if $\Delta \omega =0$.  If $Y$ has no boundary then a form is harmonic if and only if it is both closed and 
co-closed.  

Let $c\CC^p(M)$ and $\CE^p(M)$ denote the spaces of co-closed and exact forms on $M$, respectively.  Then it follows from the Hodge decomposition that  
$\Lambda^p(M) = c\CC^p(M)\oplus \CE^p(M)$. 
\subsection{The Dirichlet-to-Neumann map on differential forms} We first recall the definition of the Dirichlet-to-Neumann operator $\DN^{(p)}$ on $p$-forms defined in~\cite{Kar19}, see also 
Remark \ref{rem:whythis}. 
Given $\phi~\in~\Lambda^p(M)$, one can show that  there exists   $\omega\in\Lambda^p(\Omega)$ such that
\begin{equation}
\label{DtN_forms:eq}
\begin{cases}
\Delta\omega = 0 &\text{ in } \Omega;\\
\delta\omega = 0 &\text{ in } \Omega;\\
\mathfrak{i}^*\omega = \phi &\text{ on } M.
\end{cases}
\end{equation} 
One then sets $\DN^{(p)}(\phi) = i_nd\omega\in\Lambda^p(M)$. In~\cite{Kar19} the following properties of $\DN^{(p)}$ are proved.
\begin{theorem}
The operator $\DN^{(p)}\colon\Lambda^p(M)\to\Lambda^p(M)$ is well-defined and self-adjoint. Furthermore, one has 
\begin{enumerate}
\item[\normalfont\textrm{(a)}] $\CE^p(M)\subset \ker\DN^{(p)}$;
\item[\normalfont\textrm{(b)}] The restriction on the space of co-closed forms $\DN^{(p)}\colon c\CC^p(M)\to c\CC^p(M)$ is an operator with compact resolvent. The eigenvalues of the restriction form a sequence
\[
0\le \sigma^{(p)}_1(\Omega)\le \sigma^{(p)}_2(\Omega)\le\ldots\nearrow\infty,
\]
with the account of multiplicities.
\item[\normalfont\textrm{(c)}] The eigenvalues satisfy the following variational principle,
\[
\sigma^{(p)}_k = \inf_{E_k}\sup_{\omega\in E_k\setminus\{0\}}\frac{\|d\omega\|^2_{\Omega}}{\|\mathfrak{i}^*\omega\|^2_{M}},
\] 
where $E_k$ ranges over $k$-dimensional subspaces in $\Lambda^p(\Omega)$ satisfying $\mathfrak{i}^*E_k\subset c\CC^p(M)$.
\end{enumerate}
\end{theorem}
From now on, we will consider the Dirichlet-to-Neumann map $\DN^{(p)}$ as an operator  on $c\CC^p(M)$.
\subsection{Pohozhaev and H\"ormander type identities for differential forms} 
Let us first  prove a Pohozhaev-type identity for differential forms (cf. Theorem \ref{thm:poho}).
\begin{theorem}
\label{thm:Pohoforms}
Let $\Omega$ be a compact smooth orientable manifold with boundary $\partial \Omega  = M$.
Let $F$ be a Lipschitz vector field on $\overline{\Omega}$, and let $\omega\in \Lambda^p(\Omega)$ be a differential form satisfying $\delta d\omega=0$ in $\Omega$. Then
\[
\begin{split}
&2 \int_M\myinn{\mathfrak{i}^*(i_Fd\omega),i_nd\omega}\, \de v_M - \int_M  \myinn{F, n}\, |d\omega|^2 \, \de v_M\\ 
+&\int_\Omega|d\omega|^2\di F\, \de v_\Omega  +  \int_\Omega (\CL_Fg)[d\omega,d\omega]\, \de v_\Omega=0,
\end{split}
\]
where $\CL_F$ is a Lie derivative.
\end{theorem}
\begin{remark}
\label{rem:clarif}
 To simplify  notation,  we denote the bilinear form induced on the space of differential forms by the Riemannian metric 
 by the same letter $g$ as  the original metric on the manifold.
The integrand $(\CL_F g)[d\omega,d\omega]$ in the last term should be understood as follows: we take the Lie derivative of the Riemannian metric $g$ on $\Lambda^{p+1}(\Omega)$ 
 in the direction of $F$ and evaluate the resulting bilinear form  at the pair $[d\omega, d\omega]$. 
\end{remark}
\begin{proof}
Consider the $(d-1)$-form $\alpha:=i_Fd\omega\wedge \star d\omega$. Then by Cartan's identity $\CL_F = di_F + i_Fd$ and since $\delta d\omega = 0$ implies $d\star d \omega =0$, one has
\[
\begin{split}
d\alpha &= (di_Fd\omega)\wedge \star d\omega+ (-1)^p  i_Fd\omega\wedge(d\star d\omega) = (\CL_Fd\omega)\wedge \star d\omega \\&
= \myinn{\CL_Fd\omega,d\omega}\,\de v_\Omega =  \frac{1}{2}\left(\nabla_F|d\omega|^2 - (\CL_Fg)[d\omega,d\omega]\right)\,\de v_\Omega,
\end{split}
\]
where the last equality follows the well-known formula for the Lie derivative of a $(0,2)$-tensor (see \cite[Appendix, Theorem 50]{Peter}).
Since
\[
\di \left(|d\omega|^2 F\right) = \nabla_F|d\omega|^2 + |d\omega|^2\di F,
\]
one has 
\[
\di \left(|d\omega|^2 F\right)\,\de v_\Omega - 2d\alpha = |d\omega|^2\di F + (\CL_Fg)[d\omega,d\omega]\,\de v_\Omega.
\]
Using the Stokes and divergence theorems we obtain
\[
\begin{split}
\int_\Omega \di \left(|d\omega|^2 F\right)\,\de v_\Omega - 2d\alpha &= \int_M|d\omega|^2\myinn{F, n}\,dv_M - 2\int_M \mathfrak{i}^*(i_Fdu)\wedge \mathfrak{i}^*(*du) \\
&=  \int_M |d\omega|^2\myinn{F, n} - 2\myinn{\mathfrak{i}^*(i_Fd\omega),i_nd\omega}\, \de v_M.
\end{split}
\]
Rearranging the terms completes the proof the theorem. 
\end{proof}

Now we can prove a H\"ormander-type identity for differential forms.
\begin{theorem}\label{Poh2_forms:thm}
Let $\Omega$ be a compact smooth orientable manifold with boundary $\partial \Omega  = M$. Let $\phi\in c\CC^p(M)$ and let $F$ be a Lipschitz vector field on $\overline{\Omega}$, such that $F|_M=n$. If $\omega\in \Lambda^p(\Omega)$ is such that $i^*\omega = \phi$ and $\delta d\omega = 0$, then
\[
\myscal{\DN^{(p)}\phi,\DN^{(p)}\phi}_M - \myscal{d_M\phi,d_M\phi}_M = \int_\Omega \left(|d\omega|^2\di F - (\CL_Fg)[d\omega,d\omega]\right) \, \de v_\Omega,
\]
where $d_M$ is the differential acting on $\Lambda^p(M)$.
\end{theorem}
\begin{proof}
The result  follows from Theorem \ref{thm:Pohoforms} by noting that $|d\omega|^2 = |\DN^{(p)}(\omega)|^2 + |d_M\phi|^2$ on $M$.
\end{proof}
The following analogue of Corollary \ref{cor:hormcor1} holds.
\begin{cor} 
\label{Poh_forms:cor}
There exists a constant $C>0$, depending only on the geometry of $\Omega$ in an arbitrarily small neighbourhood of $M$, such that for any $\phi\in \Lambda^p(M)$ one has
\begin{equation}\label{Poh_forms_cor:eq}
\left|\myscal{\DN^{(p)}\phi,\DN^{(p)}\phi}_M - \myscal{d_M\phi,d_M\phi}_M\right|\le C\myscal{\DN^{(p)}\phi,\phi}_M.
\end{equation}
\end{cor}
\begin{proof}
Let $\omega$ be a solution of~\eqref{DtN_forms:eq}. Then $0=\Delta\omega = (d\delta + \delta d)\omega = \delta d\omega$. Thus, one can apply Theorem~\ref{Poh2_forms:thm}. Since $F$ is Lipschitz one has
\[
\left|\myscal{\DN^{(p)}\phi,\DN^{(p)}\phi}_M - \myscal{d_M\phi,d_M\phi}_M\right|\le  C\myscal{d\omega,d\omega}_\Omega.
\]
Since $\delta d\omega = 0$, if follows from Green's formula for differential forms (see \cite[formula (2)]{Kar19}) that
\[
\myscal{d\omega,d\omega}_\Omega=\myscal{\DN^{(p)}\phi,\phi}_M.
\]
This completes the proof of the corollary.
\end{proof}
\subsection{The Hodge Laplacian and Weyl's law for the Dirichlet-to-Neumann map}
Let $\Delta_M$ denote the Hodge Laplacian on $M$. Then $[d_M,\Delta_M] = [\delta_M,\Delta_M] = 0$. Thus, $\CE^p(M)$ and $c\CC^p(M)$ are invariant subspaces and, in particular, the restriction $\Delta_M\colon c\CC^p(M)\to c\CC^p(M)$ is a non-negative self-adjoint elliptic operator with eigenvalues
\[
0\le \widetilde{\lambda}^{(p)}_1(M)\le \widetilde{\lambda}^{(p)}_2(M)\le\ldots \nearrow\infty.
\]
The eigenvalues $\widetilde{\lambda}^{(p)}_k$ satisfy the variational principle
\[
\widetilde{\lambda}^{(p)}_k = \inf_{F_k\subset c\CC^p(M)}\sup_{\phi\in F_k\setminus \{0\}}\frac{\|d_M\phi\|^2_{M}}{\|\phi\|^2_{M}},
\]
where $F_k$ ranges over $k$-dimensional subspaces of $c\CC^p(M)$.
\begin{theorem}\label{Horm_bound:forms}
Let $\Omega$ be a compact smooth orientable Riemannian manifold with boundary $M= \partial\Omega$. Then
\[
\left|\sigma^{(p)}_k - \sqrt{\widetilde{\lambda}^{(p)}_k}\right|\le C
\]
holds with the same constant as in~\eqref{Poh_forms_cor:eq}.
\end{theorem}
\begin{proof}
The result follows from Corollary \ref{Poh_forms:cor} in the  same way as Theorem~\ref{thm:Hormbound} follows from Corollary~\ref{cor:hormcor1}.
\end{proof}
\begin{remark} 
\label{rem:whythis}
The restriction of the Hodge Laplacian to co-closed forms is an operator that has been investigated in other contexts (see \cite{JaSt07}) and has applications to physics, in particular, to the study of Maxwell equations, see \cite{BelSh08, KKL10} and references therein.
The definition  of the Dirichlet-to-Neumann map on differential forms given in \cite{Kar19} which is  used in the present paper is  inspired by the one introduced in \cite{BelSh08} (see  \cite{RaSa, JoLi} for other  definitions) and is also motivated in part by the connection to Maxwell equations. Theorem \ref{Horm_bound:forms} indicates that this definition of the  Dirichlet-to-Neumann map is  natural from the viewpoint of comparison with the boundary Hodge Laplacian. 
\end{remark}
Similarly to the proof of Theorem~\ref{thm:rough1}, one can use Theorem~\ref{Horm_bound:forms} to obtain Weyl's law for $\sigma_k^{(p)}$ from the spectral asymptotics for $\tilde{\lambda}_k^{(p)}$. \begin{theorem}
Let $\Omega$ be a compact smooth orientable manifold of dimension $d\geqslant 2$ with boundary $M=~\partial\Omega$. Then  the eigenvalue counting function for the 
Dirichlet-to-Neumann map satisfies the asymptotic relation
\begin{equation}
\label{eq:WeylHodge}
N^{(p)}(\sigma):= \#\left(\sigma^{(p)}_k<\sigma\right) = \binom{d-2}{p}\frac{\vol(\mathbb{B}^{d-1})\vol(M)}{(2\pi)^{d-1}}\sigma^{d-1} + o\left(\sigma^{d-1}\right).
\end{equation}
\end{theorem}
\begin{proof}
The theorem follows immediately from the fact that \eqref{eq:WeylHodge} holds with  $\sigma_k^{(p)}$ replaced by $\sqrt{\widetilde{\lambda}_k^{(p)}(M)}$.  As was explained to the authors by A. Strohmaier \cite{St21}, this result is essentially contained in  \cite{JaSt07, LiSt16}.
Indeed, combining the standard  Karamata Tauberian argument with \cite[formula~(1.22)]{LiSt16} giving  the heat trace asymptotics, 
 we obtain the  asymptotic formula for the counting function.  Here one takes $P$ to be the Hodge Laplacian and $A$ to be the pseudodifferential projection onto the space of co-closed forms. In order to calculate the leading term, let us apply ~\cite[formula~(1.23)]{LiSt16}. 
Let $S^*M$ be the cosphere bundle and let  $\sigma_A(\xi)\in\mathrm{End}(\Lambda^p(M))$, $\xi\in S^*M$,  be the principal symbol of $A$. As computed in~\cite[formula~(29)]{JaSt07}, 
\[
\sigma_A(\xi)[\omega] = i_{\xi^{\#}}(\xi\wedge\omega),
\]
where $\xi^\# \in SM$ is the image of $\xi$ under the musical isomorphism. 
For a fixed $x\in M$ and $\xi\in S_x^*M$, we  identify $\Lambda^p_x(M)$ with $\Lambda^p\left(\mathbb{R}^{d-1}\right)$ and set $\mathbb{R}^{d-1} = \xi\oplus\mathbb{R}^{d-2}$. This induces the decomposition
\[
\Lambda_x^p(M) \cong \left(\xi\wedge \Lambda^{p-1}\left(\mathbb{R}^{d-2}\right)\right)\oplus\Lambda^{p}\left(\mathbb{R}^{d-2}\right).
\]
It is easy to see that $\sigma_A(\xi)$ is the projection on the second summand and, thus, 
\[
\tr(\sigma_A(\xi)) =\dim \Lambda^{p}\left(\mathbb{R}^{d-2}\right) =   \binom{d-2}{p}.
\]
Integrating the trace over $\xi \in S^*M$ completes the proof.
\end{proof}

\begin{remark}
It is quite likely that the  error estimate in \eqref{eq:WeylHodge} can be improved to the bound  $O(\sigma^{d-1})$. This amounts to proving the sharp Weyl's law for $\widetilde{\lambda}^{(p)}_k$, which should be possible by further developing the techniques of~\cite{LiSt16}.

Another way to prove \eqref{eq:WeylHodge} would be to show that $\DN^{(p)}$ is an elliptic pseudodifferential operator of order one, and apply the methods 
of microlocal analysis directly to this operator. 
However, in contrast to  the Dirichlet-to-Neumann map defined in \cite{RaSa}, the proof  that the operator $\DN^{(p)}$ is pseudodifferential has  not been yet worked out in the literature (see \cite[Remark 2.4]{Kar19}).
\end{remark}
\begin{remark} One can check directly the validity of formula \eqref{eq:WeylHodge} for specific values of $p$ and $d$ for $M=\mathbb{S}^{d-1}$. 
In this case the eigenvalues of $\DN^{(p)}$ are known explicitly (see \cite[Theorem 8.1]{Kar19}) and their multiplicities coincide with the multiplicities of the corresponding eigenvalues of the Hodge Laplacian that can be found in \cite[formula (17)]{Ikeda}. It is then easy to calculate the leading term in Weyl's asymptotics using the heat trace expansion.
\end{remark}

\end{document}